\documentclass[a4paper,12pt]{article}
\usepackage{times}
\usepackage[english]{babel}
\usepackage{amssymb,amsmath}
\numberwithin{equation}{section}
\usepackage{amsthm}
\usepackage{array}
\usepackage{wasysym}
\usepackage[noadjust]{cite}
\usepackage{hyperref}
\usepackage[height=22.7cm , width = 16cm , top = 4cm , left = 3cm, a4paper]{geometry}
\usepackage{amssymb}
\usepackage{amsmath}
\usepackage{cite}
 \usepackage{pstricks}
\usepackage{epsfig}
\usepackage{verbatim}
\usepackage{multicol}
\usepackage{graphicx, color, psfrag}
\usepackage{graphics, psfrag}
\usepackage{soul,color}
\usepackage[numbers,sort]{natbib}
 \newtheorem{theorem}{Theorem}[section]

\theoremstyle{definition}

\newcommand{\e}{\end{document}}

\begin{document}

\thispagestyle{empty}

\author{
{{\bf Beih S. El-Desouky, Abdelfattah Mustafa and Shamsan AL-Garash}
\newline{\it{{}  }}
 } { }\vspace{.2cm}\\
 \small \it Department of Mathematics, Faculty of Science, Mansoura University, Mansoura 35516, Egypt.
}

\title{The Exponential Flexible Weibull Extension Distribution}

\date{}

\maketitle
\small \pagestyle{myheadings}
        \markboth{{\scriptsize The Exponential Flexible Weibull Extension}}
        {{\scriptsize {Beih S. El-Desouky,  Abdelfattah Mustafa and Shamsan AL-Garash}}}

\hrule \vskip 8pt

\begin{abstract}
This paper is devoted to study a new three- parameters model called the Exponential Flexible Weibull extension (EFWE) distribution which exhibits bathtub-shaped hazard rate. Some of it's statistical properties are obtained including ordinary and incomplete moments, quantile and generating functions, reliability and order statistics. The method of maximum likelihood is used for estimating the model parameters and the observed Fisher's information matrix is derived. We illustrate the usefulness of the proposed model by applications to real data.
\end{abstract}




\section{Introduction}
The Weibull distribution (WD) introduced by Weibull \cite{Weibull1951}, is a popular distribution for modeling lifetime data where the hazard rate function is monotone. Recently appeared new classes of distributions were based on modifications of the Weibull distribution (WD) to provide a good fit to data set with bathtub hazard failure rate Xie and Lai \cite{XieandLai1995}. Among of these, Modified Weibull (MW) distribution, Lai et al. \cite{Laietal2003} and Sarhan and Zaindin \cite{SarhanandZaindin2009}, Beta-Weibull (BW)distribution, Famoye et al. \cite{Famoyeetal2005}, Beta modified Weibull (BMW)distribution, Silva et al. \cite{Silvaetal2010} and Nadarajah et al. \cite{Nadarajahetal2011}, Kumaraswamy Weibull (KW) distribution, Cordeiro et al. \cite{Cordeiroetal2010}, Generalized modified Weibull (GMW) distribution, Carrasco et al. \cite{Carrascoetal2008} and Exponentiated modified Weibull extension (EMWE) distribution, Sarhan and Apaloo \cite{SarhanandApaloo2013}, among others.
A good review of these models is presented in Pham and Lai \cite{PhamandLai2007} and Murthy et al. \cite{Murthyetal2003}.

\noindent
The Flexible Weibull (FWE) distribution, Bebbington et al. \cite{Bebbingtonetal2007} has a wide range of applications including life testing experiments, reliability analysis, applied statistics and clinical studies.  The origin and other aspects of this distribution can be found in \cite{Bebbingtonetal2007}.
A random variable $X$ is said to have the Flexible Weibull Extension (FWE) distribution with parameters $\alpha , \beta >0$ if it's probability density function (pdf) is given by
\begin{equation} \label{1}
g(x)=(\alpha +\frac {\beta}{x^2})e^{\alpha x -\frac {\beta }{x}} \exp\left\{-e^{\alpha x -\frac{\beta }{x}}\right\},  \quad x>0,
\end{equation}

\noindent
while the cumulative distribution function (cdf) is given by
\begin{equation} \label{2}
G(x)=1-\exp\left\{-e^{\alpha x -\frac{\beta }{x}}\right\},  \quad x>0.
\end{equation}

\noindent
The survival function is given by the equation
\begin{equation} \label{3}
S(x)=1-G(x)=\exp\left\{-e^{\alpha x -\frac{\beta }{x}}\right\},  \quad x>0,
\end{equation}
and the hazard function is
\begin{equation} \label{4}
h(x)=(\alpha +\frac{\beta }{x^2})e^{\alpha x -\frac{\beta }{x}}.
\end{equation}

\noindent
In this article, a new generalization of the Flexible Weibull Extension (FWE) distribution called exponential flexible Weibull extension (EFWE) distribution is derived. Using the exponential generator applied to the odds ratio $\frac{1}{1-G(x)}$,  such as the exponential Pareto distribution by AL-Kadim and Boshi \cite{Kadim2013}, exponential lomax distribution by El-Bassiouny et al. \cite{Bassiouny2015}.  If $G(x)$ is the baseline cumulative distribution function (cdf) of a random variable, with probability density function (pdf) $g(x)$ and the exponential cumulative distribution function is

\begin{equation} \label{5}
F(x;\lambda )=1-e^{-\lambda  x}, \quad x\geq 0, \quad \lambda \geq 0.
\end{equation}

\noindent
Based on this density, by replacing $x$ with ratio $\frac{1}{1-G(x)}$. The cdf of exponential generalized distribution is defined by (see AL-Kadim and Boshi \cite{Kadim2013} and El-Bassiouny et al. \cite{Bassiouny2015})
\begin{eqnarray} \label{6}
F(x)&=&\int\limits_{0}^{\frac{1}{1-G(x)}}\lambda e^{-\lambda  t}dt
\nonumber \\
&=&
 1-\exp\left\{-\lambda \left[\frac{1}{1-G(x}\right]\right\}, \quad x\geq 0,\quad \lambda  \geq 0,
\end{eqnarray}
where $G(x)$ is a baseline cdf. Hence the pdf corresponding to Eq. (\ref{6}) is given by
\begin{equation} \label{7}
f(x)= \frac{\lambda \cdot  g(x)}{\left[1-G(x)\right]^2}\cdot \exp\left\{-\lambda \left[\frac{1}{1-G(x)}\right]\right\}.
\end{equation}

\noindent
This paper is organized as follows, we define the cumulative, density and hazard functions of the exponential flexible Weibull extension (EFWE) distribution in Section 2. In Sections 3 and 4, we introduce the statistical properties including , quantile function skewness and kurtosis, $rth$ moments and moment generating function. The distribution of the order statistics is expressed in Section 5. The maximum likelihood estimation of the parameters is determined in Section 6. Real data sets are analyzed in Section 7 and the results are compared with existing distributions. Finally, we introduce the conclusions in Section 8.


\section{The Exponential Flexible Weibull Extension Distribution}
In this section we study the three parameters Exponential Flexible Weibull Extension (EFWE) distribution. Using $G(x)$ Eq. (\ref{2}) and $g(x)$ Eq. (\ref{1}) in Eq. (\ref{6}) and Eq. (\ref{7}) to obtained the cdf and pdf of EFWE distribution. The cumulative distribution function cdf of the Exponential Flexible Weibull Extension distribution (EFWE) is given by
\begin{equation} \label{14}
F(x;\alpha ,\beta, \lambda ) = 1- \exp \left\{-\lambda e^{e^{\alpha x-\frac{\beta }{x}}} \right\}, \;  x>0, \;  \alpha, \beta, \lambda  >0.
\end{equation}

\noindent
The pdf corresponding to Eq. (\ref{14}) is given by
\begin{equation} \label{15}
f(x;\alpha ,\beta,\lambda ) = \lambda \left(\alpha +\frac{\beta }{x^2}\right) e^{\alpha x -\frac{\beta }{x}} e^{e^{\alpha x -\frac{\beta }{x}}} \exp\left\{-\lambda e^{e^{\alpha x-\frac{\beta }{x}}}\right\},
\end{equation}

\noindent
where $x>0$ and $, \alpha, \beta >0$ are two additional shape parameters.

\noindent
The survival function $S(x)$, hazard rate function $h(x)$, reversed- hazard rate function $r(x)$ and cumulative hazard rate function $H(x)$ of $X\sim EFWE(\alpha,\beta, \lambda )$ are given by
\begin{equation} \label{16}
S(x;\alpha ,\beta, \lambda  ) =1- F(x;\alpha ,\beta, \lambda )=\exp\left\{-\lambda e^{e^{\alpha x-\frac{\beta }{x}}}\right\}, \;  x>0,
\end{equation}
\begin{equation} \label{17}
h(x;\alpha ,\beta, \lambda  )= \lambda  \left(\alpha +\frac{\beta }{x^2}\right) e^{\alpha x -\frac{\beta }{x}}e^{e^{\alpha x -\frac{\beta }{x}}},
\end{equation}
\begin{equation} \label{18}
r(x;\alpha ,\beta, \lambda ) = \frac{\lambda  \left(\alpha +\frac{\beta }{x^2}\right) e^{\alpha x -\frac{\beta }{x}} e^{e^{\alpha x -\frac{\beta }{x}}} \exp \left\{-\lambda e^{e^{\alpha x-\frac{\beta }{x}}} \right\}}{1- \exp \left\{-\lambda e^{e^{\alpha x-\frac{\beta }{x}}} \right\}},
\end{equation}
\begin{equation} \label{19}
H(x;\alpha,\beta, \lambda  )= \int_0^x h(u)du= \lambda \exp \left\{e^{\alpha x-\frac{\beta }{x}} \right\},
\end{equation}

\noindent
respectively, $x>0$ and $\alpha,\beta, \lambda >0$.

\noindent
Figures 1--6 display the cdf, pdf, survival function, hazard rate function, reversed hazard rate function and cumulative hazard rate function of the EFWE($\alpha$, $\beta$, $\lambda $) distribution for some parameter values.

\begin{center}
\includegraphics[width=10cm,height=6cm]{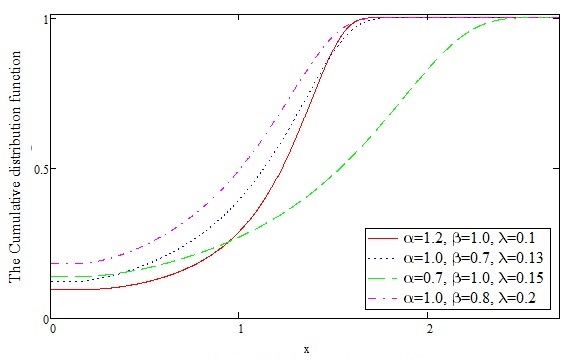}\\
Figure 1: The cdf of the EFWE distribution for different values of parameters.
\end{center}

\vspace{0.4 cm}
\begin{center}
\includegraphics[width=10cm,height=6cm]{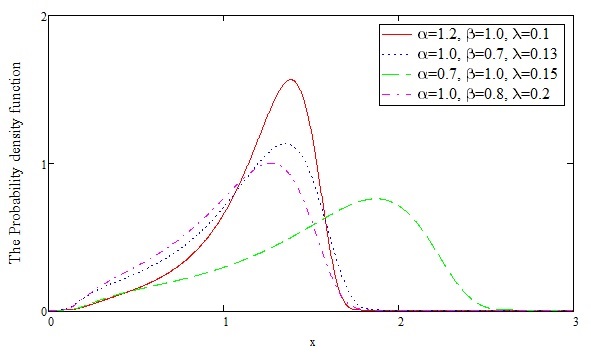}\\
Figure 2: The pdf of the EFWE distribution for different values of parameters.
\end{center}

\vspace{0.4 cm}

\begin{center}
\includegraphics[width=10cm,height=6cm]{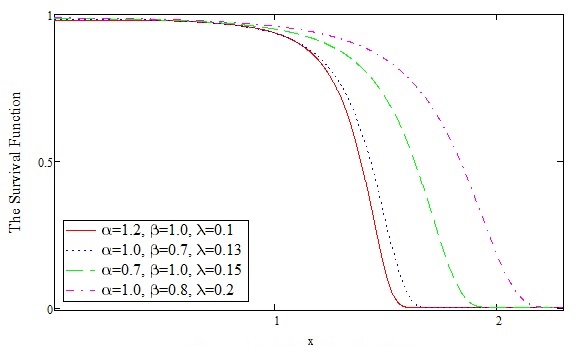}\\
Figure 3: The survival function of the EFWE distribution for different values of parameters.
\end{center}

\vspace{0.4 cm}

\begin{center}
\includegraphics[width=10cm,height=6cm]{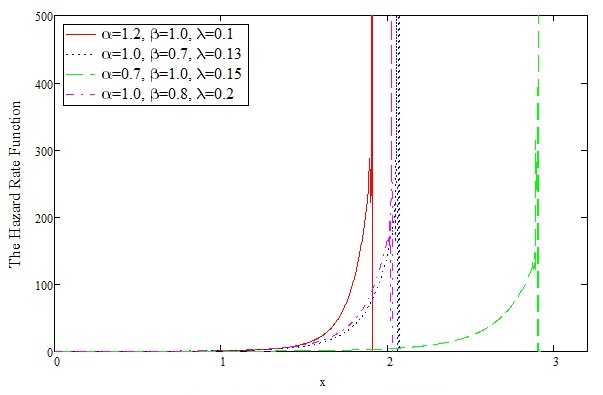}\\
Figure 4: The hazard rate function of the EFWE distribution for different values of parameters.
\end{center}

\vspace{0.4 cm}
\begin{center}
\includegraphics[width=10cm,height=6cm]{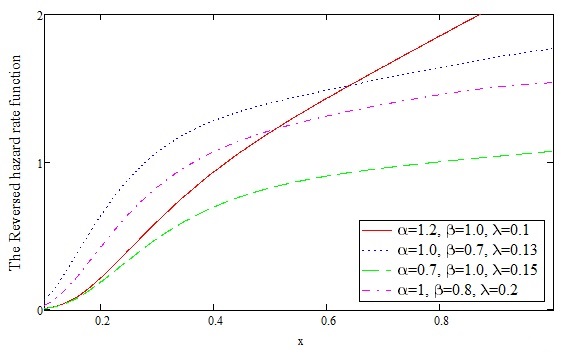}\\
Figure 5: The reversed hazard rate function of the EFWE distribution  for different values of parameters.
\end{center}

\vspace{0.4 cm}
\begin{center}
\includegraphics[width=10cm,height=6cm]{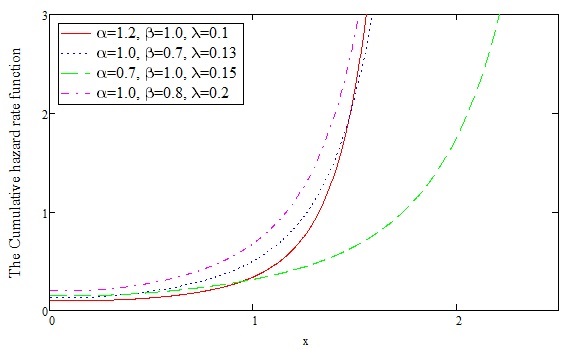}\\
Figure 6: The cumulative hazard rate function of the EFWE distribution for different values of parameters.
\end{center}

\section{Statistical Properties}
In this section, we study the statistical properties for the EFWE distribution, specially quantile function and simulation median, skewness, kurtosis and moments.

\subsection{Quantile and simulation}
The quantile $x_q$ of the EFWE ($\alpha,\beta, \lambda $) distribution random variable is given by
\begin{equation} \label{20}
F(x_q)=q, \quad 0<q<1.
\end{equation}

\noindent
Using the distribution function of EFWE distribution, from (\ref{14}), we have
\begin{equation} \label{21}
\alpha x_q^2-k(q) x_q-\beta=0,
\end{equation}
where
\begin{equation} \label{22}
k(q)=\ln\left\{\ln \left(-\frac{\ln(1-q)}{\lambda }\right) \right\}.
\end{equation}

\noindent
So, the simulation of the EFWE distribution random variable is straightforward. Let $U$ be a uniform variate on the unit interval $(0,1)$. Thus, by means of the inverse transformation method, we consider the random variable $X$ given by
\begin{equation} \label{23}
X=\frac{k(u)\pm\sqrt{k(u)^2+4\alpha \beta}}{2\alpha}.
\end{equation}

\noindent
Since the median is $50$ \% quantile then by setting $q=0.5$ in Eq. (\ref{21}), can be obtained the median $M$ of EFWE distribution.

\subsection{The Mode of EFWE}
In this subsection, we will derive the mode of the EFWE($\alpha,\beta, \lambda $) distribution by deriving its pdf with respect to $x$ and equal it to zero thus the mode of the EFWE($\alpha,\beta, \lambda $) distribution can be obtained as a nonnegative solution of the following nonlinear equation
\begin{eqnarray} \label{mod9}
&&
\lambda e^{\alpha x- \frac{\beta}{x}} e^{e^{\alpha x- \frac{\beta}{x}}} \exp\left\{-\lambda e^{e^{\alpha x- \frac{\beta}{x}}} \right\}
\left[-\frac{2\beta}{x^3}+\left(\alpha+\frac{\beta}{x^2} \right)
\left(1+ e^{\alpha x- \frac{\beta}{x}}-\lambda e^{e^{\alpha x- \frac{\beta}{x}}} \right) \right]=0
\end{eqnarray}%

\noindent
From Figure \ref{fig2}, the pdf for EFWE distribution has only one peak, It is a unimodal distribution, so the above equation has only one solution. It is not possible to get an explicit solution of (\ref{mod9}) in the general case. Numerical methods should be used such as bisection or fixed-point method to solve it. Some values of median and mode for various values of parameters $\alpha, \beta$ and $\lambda$ calculated in Table 1.

\begin{center}
Table 1: The median and mode for EFWE($\alpha, \beta, \lambda$)\\
\begin{tabular}{ccccc} \hline
$\alpha$ & $\beta$ & $\lambda$ & median & mode \\ \hline
0.015   & 0.381   & 0.076  &   53.3576   &   10.6657  \\
0.158   & 0.158   & 0.273  &  0.801066   &  1.96923\\
0.700   & 1.000   & 0.150  &  1.537340   &  1.87122\\
1.000   & 0.700   & 0.130  &  1.132920   & 1.35312 \\
1.000   & 0.800   & 0.200  &  1.009750   & 1.27259 \\
1.200   & 1.000   & 0.100  &  1.228750   & 1.38465  \\
\hline
\end{tabular}
\end{center}

\subsection{The Skewness and Kurtosis}
The analysis of the variability Skewness and Kurtosis on the shape parameters $\alpha,\beta$  can be investigated based on quantile measures. The short comings of the classical Kurtosis measure are well-known. The Bowely's skewness based on quartiles is given by, Kenney and Keeping \cite{KenneyandKeeping1962}
\begin{equation} \label{24}
S_k=\frac{ q_{(0.75)} -2 q_{(0.5)}+q_{(0.25)}}{q_{(0.75)}- q_{(0.25)}},
\end{equation}

\noindent
and the Moors Kurtosis based on quantiles, Moors \cite{Moors1998}
\begin{equation} \label{25}
K_u=\frac{ q_{(0.875)} - q_{(0.625)}-q_{(0.375)}+ q_{(0.125)}}{q_{(0.75)}- q_{(0.25)}},
\end{equation}

\noindent
where $q_{(.)}$ represents quantile function.

\subsection{The Moments}
In this subsection we discuss the $rth$ moment for EFWE distribution. Moments are important in any statistical analysis, especially in applications. It can be used to study the most important features and characteristics of  a distribution (e.g. tendency, dispersion, skewness and kurtosis).
\begin{theorem} \label{Th1}
If $X$ has EFWE $(\alpha ,\beta, \lambda )$ distribution, then the $r$th moments of random variable $X$, is given by the following
\begin{equation}  \label{26}
\mu_ r^{'}= \sum_{i=0}^{\infty}\sum_{j=0}^{\infty}\sum_{k=0}^{\infty}
\frac{(-1)^{i+k} \lambda ^{i+1}\beta ^k (i+1)^j(j+1)^k}
{i! j! k!} \left[\frac{\Gamma(r-k+1)}{\alpha ^{r-k} (j+1)^{r-k+1}}+\frac{\beta \Gamma(r-k-1)}{\alpha ^{r-k-1} (j+1)^{r-k-1}}\right].
\end{equation}
\end{theorem}
\begin{proof}
We start with the well known distribution of the $r$th moment of the random variable $X$ with probability density function $f(x)$ given by
\begin{equation} \label{27}
\mu_ r^{'}=\int\limits_{0}^{\infty} x^r f(x;\alpha ,\beta, \lambda  ) dx.
\end{equation}

\noindent
Substituting from Eq. (\ref{15}) into Eq. (\ref{27}) we get
\begin{equation*}
\mu_ r^{'}= \int_0^\infty x^r \lambda \left(\alpha +\frac{\beta }{x^2}\right) e^{\alpha x -\frac{\beta }{x}} e^{e^{\alpha x -\frac{\beta }{x}}} \exp\left\{-\lambda e^{e^{\alpha x-\frac{\beta }{x}}}\right\} dx,
\end{equation*}

\noindent
 the expansion of
$\exp\left\{-\lambda e^{e^{\alpha x-\frac{\beta }{x}}}\right\}$ is
$$
e^{-\lambda e^{e^{\alpha x-\frac{\beta }{x}}}}=\sum_{i=0}^{\infty}\frac{(-1)^i \lambda ^i}{i!} e^{i e^{\alpha x -\frac{\beta }{x}}},
$$
then we get
\begin{equation*}
\mu_ r^{'}= \sum_{i=0}^{\infty} \frac{(-1)^i \lambda ^{i+1}}{i!} \int_0^{\infty} x^r \left(\alpha  +\frac{\beta }{x^2}\right) e^{\alpha x - \frac{\beta }{x}} e^{(i+1)e^{\alpha x - \frac{\beta }{x}}} dx,
\end{equation*}
 \noindent
using series expansion of $e^{(i+1) e^{\alpha x - \frac{\beta }{x}}}$,
 \[
 e^{(i+1) e^{\alpha x - \frac{\beta }{x}}} =\sum_{j=0}^{\infty}\frac{(i+1)^j}{j!} e^{j(\alpha x -\frac{\beta }{x})},
 \]
we obtain
\begin{eqnarray*}
\mu_ r^{'} & = & \sum_{i=0}^{\infty}\sum_{j=0}^{\infty} \frac{(-1)^i \lambda ^{i+1}(i+1)^j}{i! j!} \int_0^{\infty} x^r \left(\alpha  +\beta x^{-2}\right) e^{(j+1)(\alpha x - \frac{\beta }{x})} dx,
\\
&= & \sum_{i=0}^{\infty}\sum_{j=0}^{\infty} \frac{(-1)^i \lambda ^{i+1}(i+1)^j}{i! j!}  \int_0^{\infty} x^r \left(\alpha  +\beta x^{-2}\right) e^{(j+1)\alpha  x} e^{-(j+1)\frac{\beta }{x}} dx,
\end{eqnarray*}

\noindent
using series expansion of  $e^{-(j+1)\frac{\beta }{x}}$,
\[
e^{-(j+1)\frac{\beta }{x}}= \sum_{k=0}^{\infty}\frac{(-1)^k (j+1)^k \beta ^k }{k!} x^{-k},
\]
we have
\begin{eqnarray*}
\mu_ r^{'} & = & \sum_{i=0}^{\infty}\sum_{j=0}^{\infty}\sum_{k=0}^{\infty} \frac{(-1)^{i+k} \lambda ^{i+1} \beta ^k (i+1)^j (j+1)^k}{i! j! k!}\int_0^{\infty} x^{r-k} \left(\alpha  +\frac{\beta}{ x^2}\right) e^{(j+1)\alpha x}  dx,
\\
 &= & \sum_{i=0}^{\infty}\sum_{j=0}^{\infty}\sum_{k=0}^{\infty} \frac{(-1)^{i+k} \lambda ^{i+1} \beta ^k (i+1)^j (j+1)^k}{i! j! k!} \left[\int_0^{\infty} \alpha x^{r-k} e^{(j+1)\alpha x}  dx
\right.
\\
&& \left.
+ \int_0^{\infty}\beta x^{r-k-2} e^{(j+1)\alpha x}  dx\right],
\end{eqnarray*}

\noindent
by using the definition of gamma function ( Zwillinger \cite{Zwillinger2014}), in the form,
\[
\Gamma (z)=x^z \int_0^{\infty}e^{tx} t^{z-1}dt, \quad z,x,>0.
\]
Finally, we obtain the $r$th moment of EFWE distribution in the form
\begin{equation*}
\mu_ r^{'}= \sum_{i=0}^{\infty}\sum_{j=0}^{\infty}\sum_{k=0}^{\infty}
\frac{(-1)^{i+k} \lambda ^{i+1}\beta ^k (i+1)^j(j+1)^k}
{i! j! k!}\times \left[\frac{\Gamma(r-k+1)}{\alpha ^{r-k} (j+1)^{r-k+1}}+\frac{\beta \Gamma(r-k-1)}{\alpha ^{r-k-1} (j+1)^{r-k-1}}\right].
\end{equation*}
This completes the proof.

\end{proof}

\section{The Moment Generating Function}
The moment generating function (mgf) $M_X (t)$ of a random variable $X$ provides the basis of an alternative route to analytic results compared with working directly with the pdf and cdf of $X$.
\begin{theorem}
The moment generating function (mgf) of EFWE distribution is given by
\begin{eqnarray}
M_X(t) & = & \sum_{i=0}^{\infty}\sum_{j=0}^{\infty}\sum_{k=0}^{\infty}\sum_{r=0}^{\infty}
\frac{(-1)^{i+k} \lambda ^{i+1}\beta ^k (i+1)^j(j+1)^k t^r}{i! j! k! r!}  \times
\nonumber\\
&& \hspace{3 cm}
\left[\frac{\Gamma(r-k+1)}{\alpha ^{r-k} (j+1)^{r-k+1}}+\frac{\beta \Gamma(r-k-1)}{\alpha ^{r-k-1} (j+1)^{r-k-1}}\right].
\end{eqnarray}
\end{theorem}
\begin{proof}
The moment generating function of the random variable $X$ with probability density function $f(x)$ is given by
\begin{equation} \label{28}
M_X(t)=\int\limits_{0}^{\infty} e^{tx} f(x) dx.
\end{equation}

\noindent
using series expansion of $e^{tx}$, we obtain
\begin{eqnarray} \label{28a}
M_X(t) & = & \sum_{r=0}^\infty \frac{t^r}{r!} \int_0^{\infty} x^r f(x) dx.
\nonumber\\
& = &
\sum_{r=0}^\infty \frac{t^r}{r!} \mu_r^{'}.
\end{eqnarray}

\noindent
Substituting from Eq. (\ref{26}) into Eq. (\ref{28a}) we obtain the moment generating function (mgf) of EFWE distribution in the form
\begin{eqnarray*}
M_X(t) & = & \sum_{i=0}^{\infty}\sum_{j=0}^{\infty}\sum_{k=0}^{\infty}\sum_{r=0}^{\infty}
\frac{(-1)^{i+k} \lambda ^{i+1}\beta ^k (i+1)^j(j+1)^k t^r}
{i! j! k! r!}\times
\\
&& \hspace{3cm}
\left[\frac{\Gamma(r-k+1)}{\alpha ^{r-k} (j+1)^{r-k+1}}+\frac{\beta \Gamma(r-k-1)}{\alpha ^{r-k-1} (j+1)^{r-k-1}}\right].
\end{eqnarray*}

\noindent
This completes the proof.
\end{proof}

\section{Order Statistics }
In this section, we derive closed form expressions for the probability density function of the $r$th order statistic of the EFWE distribution.
Let $X_{1:n},X_{2:n},\cdots,X_{n:n}$ denote the order statistics obtained from a random sample $X_1$, $X_2$, $\cdots$, $X_n$ which taken from a continuous population with cumulative distribution function cdf $F(x;\varphi )$ and probability density function pdf $f(x;\varphi )$, then the probability density function of $X_{r:n}$ is given by

\begin{equation} \label{29}
f_{r:n}(x;\varphi )=\frac{1}{B(r,n-r+1)}\left[F(x;\varphi )\right]^{r-1} \left[1-F(x;\varphi )\right]^{n-r} f(x;\varphi),
\end{equation}

\noindent
where $f(x;\varphi)$, $F(x;\varphi)$ are the pdf and cdf of EFWE $(\alpha ,\beta, \lambda)$ distribution given by Eq. (\ref{15}) and Eq. (\ref{14}) respectively, $\varphi =(\alpha ,\beta, \lambda  )$ and $B(.,.)$ is the Beta function, also we define first order statistics $X_{1:n}= \min(X_1,X_2,\cdots,X_n)$, and the last order statistics as $X_{n:n}= \max (X_1,X_2,\cdots,X_n)$. Since $0 < F(x;\varphi )< 1$  for $x>0$, we can use the binomial expansion of $[1-F(x;\varphi )]^{n-r}$ given as follows
\begin{equation} \label{30}
\left[1-F(x;\varphi )\right]^{n-r}=\sum_{i=0}^{n-r}\begin{pmatrix} n-r \\ i \end{pmatrix} (-1)^i [F(x;\varphi )]^i.
\end{equation}

\noindent
Substituting from Eq. (\ref{30}) into Eq. (\ref{29}), we obtain
\begin{equation} \label{31}
f_{r:n}(x;\varphi )=\frac{1}{B(r,n-r+1)} f(x;\varphi )\sum_{i=0}^{n-r}\begin{pmatrix} n-r \\ i \end{pmatrix} (-1)^i \left[F(x;\varphi )\right]^{i+r-1}.
\end{equation}

\noindent
Substituting from Eq. (\ref{14}) and Eq. (\ref{15}) into Eq. (\ref{31}), we obtain
\begin{equation} \label{32}
f_{r:n}(x;\varphi)=\sum_{i=0}^{n-r} \frac{(-1)^i n!}{i! (r-1)! (n-r-i)!} \left[F(x,\varphi)\right]^{i+r-1} f(x;\varphi).
\end{equation}

\noindent
Relation (\ref{32}) shows that $f_{r:n}(x;\varphi )$ is the weighted average of the Exponential Flexible Weibull Extension distribution withe different shape parameters.

\section{Parameters Estimation}
In this section, point and interval estimation of the unknown parameters of the EFWE distribution are derived by using the method of maximum likelihood based on a complete sample.

\subsection{Maximum likelihood estimation}
Let $x_1,x_2,\cdots ,x_n$ denote a random sample of complete data from the EFWE distribution. The Likelihood function is given as
\begin{equation} \label{33}
L = \prod_{i=1}^{n} f(x_i;\alpha, \beta, \lambda ),
\end{equation}

\noindent
substituting from (\ref{15}) into (\ref{33}), we have
\begin{equation*}
L = \prod_{i=1}^{n}\lambda \left(\alpha +\frac{\beta }{x^2}\right) e^{\alpha x -\frac{\beta }{x}} e^{e^{\alpha x -\frac{\beta }{x}}} \exp\left\{-\lambda e^{e^{\alpha x-\frac{\beta }{x}}}\right\}.
\end{equation*}

\noindent
The log-likelihood function is
\begin{equation} \label{34}
\mathcal{L} = n \ln(\lambda ) + \sum_{i=1}^{n} \ln (\alpha +\frac{\beta }{x_i^2}) +  \sum_{i=1}^{n} \left(\alpha x_i -\frac{\beta }{x_i}\right) + \sum_{i=1}^{n} e^{\alpha x_i -\frac{\beta }{x_i}} -\lambda \sum_{i=1}^{n} e^{e^{\alpha x_i-\frac{\beta }{x_i}}}.
\end{equation}

\noindent
The maximum likelihood estimation of the parameters are obtained by differentiating the log-likelihood function $\mathcal{L}$ with respect to the parameters $\alpha, \beta $ and $\lambda  $ and setting the result to zero, we have the following normal equations.
\begin{eqnarray} \label{35}
\frac{\partial \mathcal{L}}{\partial \alpha } & = &
\sum_{i=1}^{n}\frac{ x_i^2}{\beta+\alpha x_i^2 }+\sum_{i=1}^{n} x_i +\sum_{i=1}^{n} x_i e^{\alpha x_i-\frac{\beta }{x_i}} - \lambda  \sum_{i=1}^{n} x_i e^{\alpha x_i-\frac{\beta }{x_i}}
e^{e^{\alpha x_i-\frac{\beta }{x_i}}} = 0,
\\ \label{36}
\frac{\partial \mathcal{L}}{\partial \beta  } & = &
\sum_{i=1}^{n}\frac{1}{\beta+\alpha x_i^2}-\sum_{i=1}^{n} \frac{1}{x_i} - \sum_{i=1}^{n} \frac{1}{x_i} e^{\alpha x_i-\frac{\beta }{x_i}} + \lambda  \sum_{i=1}^{n} \frac{1}{x_i}e^{\alpha x_i-\frac{\beta }{x_i}}e^{e^{\alpha x_i-\frac{\beta }{x_i}}} = 0,
\\ \label{37}
\frac{\partial \mathcal{L} }{\partial \lambda } &=&
\frac{n}{\lambda }- \sum_{i=1}^{n} e^{e^{\alpha x_i-\frac{\beta }{x_i}}}= 0.
\end{eqnarray}

\noindent
The MLEs can be obtained by solving the nonlinear equations previous, (\ref{35})-(\ref{37}), numerically for $\alpha, \beta$ and $\lambda $.

\subsection{Asymptotic confidence bounds}
In this section, we derive the asymptotic confidence intervals when $\alpha, \beta >0$ and $\lambda >0$ as the MLEs of the unknown parameters $\alpha, \beta >0$ and $\lambda>0$ can not be obtained in closed forms, by using variance covariance matrix $I^{-1}$ see Lawless \cite{Lawless2003}, where $I^{-1}$ is the inverse of the observed information matrix which defined as follows
\begin{eqnarray} \label{39}
\mathbf{I^{-1}} &= &
\left(
\begin{array}{ccc}
-\frac{\partial ^2 \mathcal{L}}{\partial \alpha  ^2} & -\frac{\partial ^2 \mathcal{L}}{\partial \alpha  \partial \beta } & -\frac{\partial ^2 \mathcal{L}}{\partial \alpha \partial \lambda  }
\\
-\frac{\partial ^2 \mathcal{L}}{\partial \beta  \partial \alpha } & -\frac{\partial ^2 \mathcal{L}}{\partial \beta^2} & -\frac{\partial ^2 \mathcal{L}}{\partial \beta \partial \lambda  }
\\
-\frac{\partial ^2 \mathcal{L}}{\partial \lambda  \partial \alpha } & -\frac{\partial ^2 \mathcal{L}}{\partial \lambda  \partial \beta } & -\frac{\partial ^2 \mathcal{L}}{\partial \lambda ^2}
\end{array}
\right)^{-1}
 =
\left(
\begin{array}{ccc}
var(\hat{\alpha }) & cov( \hat{\alpha }, \hat{\beta }) & cov( \hat{\alpha }, \hat{ \lambda  })
\\
cov( \hat{\beta },\hat{\alpha  }) & var( \hat{\beta }) & cov( \hat{\beta }, \hat{ \lambda  })
\\
cov( \hat{ \lambda  }, \hat{\alpha }) & cov( \hat{ \lambda }, \hat{\beta }) &  var( \hat{ \lambda })
\end{array}
\right).
\end{eqnarray}

\noindent
The second partial derivatives included in   $I$  are given as follows.
\begin{eqnarray} \label{40}
\frac{\partial ^2 \mathcal{L}}{\partial \alpha^2 } &= &
-\sum_{i=1}^{n}\frac{x_i^4}{\left(\beta+\alpha x_i^2\right)^2} +\sum_{i=1}^{n} x_i^2 e^{\alpha x_i -\frac{\beta }{x_i}} - \lambda \sum_{i=1}^n x_i^2 e^{\alpha x_i-\frac{\beta }{x_i}}e^{e^{\alpha x_i-\frac{\beta }{x_i}}} \left[1+e^{\alpha x_i -\frac{\beta }{x_i}} \right],
\\ \label{41}
\frac{\partial ^2 \mathcal{L}}{\partial \alpha \partial \beta } &= &
-\sum_{i=1}^{n}\frac{x_i^2}{\left(\beta+\alpha x_i^2\right)^2} -\sum_{i=1}^{n} e^{\alpha x_i -\frac{\beta }{x_i}} + \lambda \sum_{i=1}^n e^{\alpha x_i-\frac{\beta }{x_i}}e^{e^{\alpha x_i-\frac{\beta }{x_i}}} \left[1+e^{\alpha x_i -\frac{\beta }{x_i}} \right],
\\ \label{42}
\frac{\partial ^2 \mathcal{L}}{\partial \alpha  \partial \lambda  } &= &
-\sum_{i=1}^{n} x_i e^{\alpha x_i-\frac{\beta }{x_i}}e^{e^{\alpha x_i-\frac{\beta }{x_i}}},
\\ \label{43}
\frac{\partial ^2 \mathcal{L}}{\partial \beta^2 } &= &
-\sum_{i=1}^{n}\frac{1}{\left(\beta+\alpha x_i^2\right)^2} +\sum_{i=1}^{n}\frac{1}{x_i^2} e^{\alpha x_i -\frac{\beta }{x_i}} - \lambda   \sum_{i=1}^n \frac{1}{x_i^2} e^{\alpha x_i-\frac{\beta }{x_i}}e^{e^{\alpha x_i-\frac{\beta }{x_i}}} \left [1+e^{\alpha x_i -\frac{\beta }{x_i}} \right],
\nonumber\\
&&
\\ \label{44}
\frac{\partial ^2 \mathcal{L}}{\partial \beta  \partial \lambda  }&= &
\sum_{i=1}^{n}\frac{1}{x_i} e^{\alpha x_i-\frac{\beta }{x_i}}e^{e^{\alpha x_i-\frac{\beta }{x_i}}} ,
 \\ \label{45}
\frac{\partial ^2 \mathcal{L}}{\partial \lambda ^2} &=& -\frac{n}{\lambda ^2}.
\end{eqnarray}

\noindent
We can derive the $(1-\delta)100\%$ confidence intervals of the parameters $\alpha, \beta$ and $\lambda$, by using variance matrix as in the following forms
$$ \hat{\alpha} \pm Z_{\frac{\delta}{2}}\sqrt{var(\hat{\alpha })},\quad \hat{\beta} \pm Z_{\frac{\delta}{2}}\sqrt{var(\hat{\beta})}, \quad \hat{\lambda } \pm Z_{\frac{\delta}{2}}\sqrt{var(\hat{\lambda })}, $$
where $Z_{\frac{\delta}{2}}$ is the upper $(\frac{\delta}{2})$-th percentile of the standard normal distribution.


\section{Application}
In this section, we present the analysis of a real data set using the EFWE $(\alpha, \beta, \lambda )$ model and compare it with the other fitted models like A flexible Weibull extension (FWE) distribution, Weibull distribution (WD) , linear failure rate distribution (LFRD), exponentiated Weibull distribution(EWD), generalized linear failure rate distribution (GLFRD) and  exponentiated flexible Weibull distribution (EFWD) using Kolmogorov Smirnov (K-S) statistic, as well as Akaike information criterion (AIC), \cite{Akaike1974}, Akaike Information Citerion with correction (AICC) and Bayesian information criterion (BIC) \cite{Schwarz1978} values.

\noindent
 Consider the data have been obtained from Aarset \cite{Aarset1987}, and widely reported in many literatures. It represents the lifetimes of 50 devices, and also, possess a bathtub-shaped failure rate property, Table 2.

\begin{center}
Table 2: Life time of 50 devices, see Aarset \cite{Aarset1987}.\\
\begin{tabular}{ccccccccccccccccc} \hline
0.1	& 0.2 & 1  & 1	& 1  & 1  & 1  & 2  & 3  & 6 \\
7   & 11  & 12 & 18	& 18 & 18 & 18 & 18 & 21 & 32 \\	
36  & 40  & 45 & 46 & 47 & 50 & 55 & 60 & 63 & 63 \\	
67  & 67  & 67 & 67 & 72 & 75 & 79 & 82 & 82 & 83 \\	
84  & 84  & 84 & 85 & 85 & 85 & 85 & 85 & 86 & 86 & \\	\hline
\end{tabular}
\end{center}

\noindent
Table 3 gives MLEs of parameters of the EFWE distribution and K-S Statistics. The values of the log-likelihood functions, AIC, AICC and BIC are in Table 4.

\begin{center}
Table 3: MLEs and K--S of parameters for Aarset data \cite{Aarset1987}.\\
\begin{tabular}{lcccccccc} \hline
Model                    & MLE of the parameters  &  K-S & P-value \\ \hline
FW($\alpha, \beta$ )	 & $\hat{\alpha}$ = 0.0122, $\hat{\beta}$ = 0.7002 & 	0.4386   & 4.29 $\times 10^{-9}$  \\
W($\alpha, \beta$ )    & $\hat{\alpha}$ = 44.913, $\hat{\beta}$ = 0.949 & 0.2397   & 0.0052 \\
LFR (a, b)	         &$\hat{a}$ = 0.014, $\hat{b}$ = 2.4 $\times 10^{-4}$ & 	0.1955   & 0.0370 \\
EW($\alpha, \beta, \gamma $) & $\hat{\alpha}$ = 91.023, $\hat{\beta}$ = 4.69, $\hat{\gamma}$ = 0.164 & 0.1841  & 0.0590 \\
GLFR(a, b, c)    & $\hat{a}$ = 0.0038, $\hat{b}$ = 3.04 $\times 10^{-4}$, $\hat{c}$ = 0.533  & 	0.1620   & 0.1293 \\
EFW($\alpha , \beta , \theta$ )& $\hat{\alpha}$ = 0.0147, $\hat{\beta}$ = 0.133, $\hat{\theta}$ = 4.22  & 0.1433 & 0.2617  \\
EFWE($\alpha , \beta , \lambda$ ) & $\hat{\alpha}$ = 0.015, $\hat{\beta}$ = 0.381, $\hat{\lambda }$ = 0.076 & 0.13869 & 0.2719\\ \hline
\end{tabular}
\end{center}

\vspace{0.4 cm}

\begin{center}
Table 4: Log-likelihood, AIC, AICC and BIC values of models fitted for Aarset data \cite{Aarset1987}.\\
\begin{tabular}{lccccccc} \hline
Model	                 & $\mathcal{L}$ &-2 $\mathcal{L}$ &	AIC	& AICC	& BIC   \\ \hline
FW($\alpha, \beta $)	            & -250.810 & 501.620 & 505.620 & 505.88	 & 509.448   \\
W($\alpha, \beta$ ) 	            & -241.002 & 482.004 & 486.004 & 486.26  & 489.828 \\
LFR (a, b)                          & -238.064 & 476.128 & 480.128 & 480.38  & 483.952 \\
EW($\alpha, \beta, \gamma$ ) 	    & -235.926 & 471.852 & 477.852 & 478.37  & 483.588 \\
GLFR(a, b, c)                       & -233.145 & 466.290 & 472.290 & 472.81	 & 478.026 \\
EFW($\alpha , \beta , \theta$ )     & -226.989 & 453.978 & 459.979 & 460.65	 & 465.715 \\
EFWE($\alpha , \beta , \lambda$ )   & -224.832 & 449.664 & 455.664 & 456.19  & 461.400 \\ \hline
\end{tabular}
\end{center}

\noindent
We find that the EFWE distribution with the three-number of parameters provides a better fit than the previous new modified
a flexible Weibull extension distribution(FWE) which was the best in Bebbington et al. \cite{Bebbingtonetal2007}. It has the largest likelihood, and the smallest K-S, AIC, AICC and BIC values among those considered in this paper.\\

\noindent
Substituting the MLE's of the unknown parameters $ \alpha, \beta, \lambda  $  into (\ref{39}), we get estimation of the variance covariance matrix as the following

$$
I_0^{-1}=\left(
\begin{array}{rrrr}

1.11 \times 10^{-6}	 & -1.175 \times 10^{-5}	 & -2.187\times 10^{-5} \\
-1.175\times 10^{-5}	& 0.021	 & 3.275\times 10^{-4} \\
-2.187\times 10^{-5}	& 3.275\times10^{-4}	& 5.469\times 10^{-4}
\end{array}
\right)
$$

\noindent
The approximate 95\% two sided confidence intervals of the unknown parameters $\alpha, \beta$ and $\lambda $ are $\left[
0.013,	0.017\right]$, $\left[0.1, 0.662\right]$ and  $\left[0.03, 0.122\right]$, respectively.\\

\noindent
To show that the likelihood equation have unique solution, we plot the profiles of the log-likelihood function of $\alpha,  \beta$ and  $\lambda $ in Figures 7 and 8.

\begin{center}
\includegraphics[width=7.2cm,height=6.2cm]{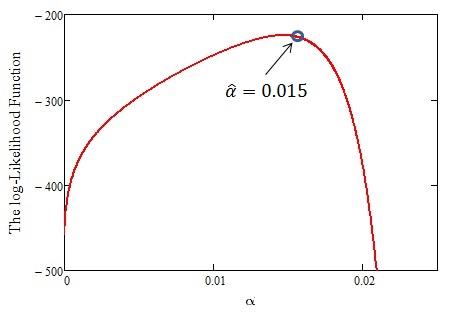}
\includegraphics[width=7.2cm,height=6.2cm]{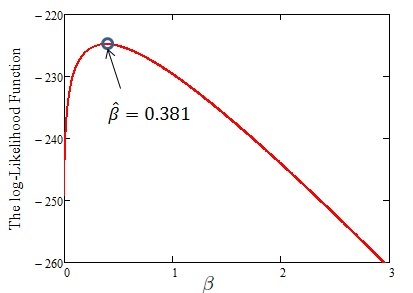} \\
Figure 7: The profile of the log-likelihood function of $\alpha, \beta$.
\end{center}

\vspace{0.2 cm}

\begin{center}
\includegraphics[width=7.2cm,height=6.2cm]{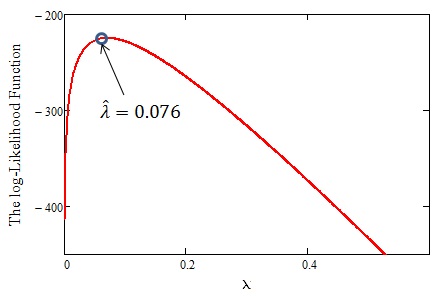}\\
Figure 8: The profile of the log-likelihood function of $\lambda$.
\end{center}

\noindent
The nonparametric estimate of the survival function using the Kaplan-Meier method and its fitted parametric estimations when the distribution is assumed to be $FW, W, LFR, EW, GLFR, EFW$ and $EFWE$ are computed and plotted in Figure 9.

\begin{center}
\includegraphics[width=10cm,height=6cm]{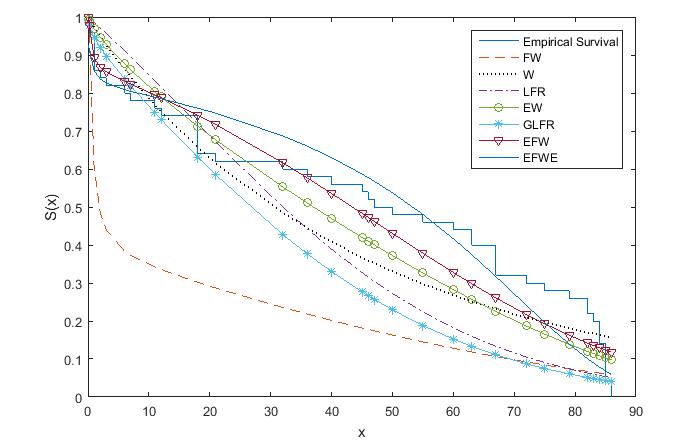}\\
Figure 9: The Kaplan-Meier estimate of the survival function for the data.
\end{center}

\noindent
Figures 10 and 11 give the form of the hazard rate and CDF for the $FW, W, LFR, EW, GLFR, EFW$ and $EFWE$  which are used to fit the data after replacing the unknown parameters included in each distribution by their MLE.

\begin{center}
\includegraphics[width=10cm,height=6cm]{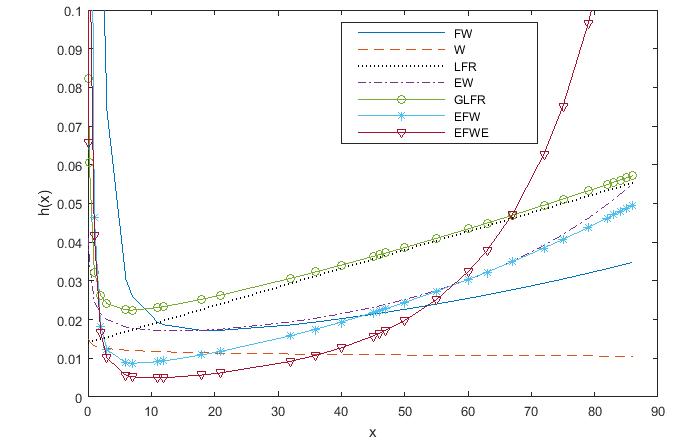}\\
Figure 10:  The Fitted hazard rate function for the data.
\end{center}

\vspace{0.4 cm}

\begin{center}
\includegraphics[width=10cm,height=6cm]{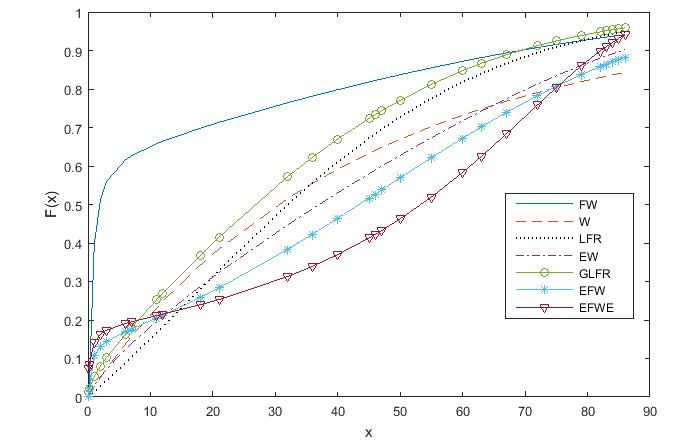}\\
Figure 11: The Fitted cumulative distribution function for the data.
\end{center}

\section{Conclusions}
A new distribution, based on  exponential generalized distributions, has been proposed and its properties are studied. The idea is to add parameter to a flexible Weibull extension distribution, so that the hazard function is either increasing or more importantly, bathtub shaped. Using  Weibull  generator component, the distribution has flexibility to model the second peak in a distribution. We have shown that the exponential flexible Weibull extension  EFWE distribution fits certain well-known data sets better than existing modifications of the exponential generalized family of probability distribution.



\begin{thebibliography}{99}

\bibitem{Aarset1987}
Aarset, M. V. (1987). "How to identify bathtub hazard rate". {\em IEEE Transactions on
Reliability}, 36, 106--108.

\bibitem{Akaike1974}
Akaike, H. (1974). A new look at the statistical model identification. {\em IEEE Transactions on Automatic Control}, AC-19, 716--23.

\bibitem{Bebbingtonetal2007}
Bebbington, M. S., Lai, C. D. and Zitikis, R. (2007). A flexible Weibull extension. {\em Reliability Engineering \& System Safety }, 92(6), 719--26.

\bibitem{Carrascoetal2008}
Carrasco M., Ortega E. M. and Cordeiro G.M. (2008). A generalized modified Weibull distribution for lifetime modeling. {\em Computational Statistics and Data Analysis}, 53(2), 450--62.

\bibitem{Cordeiroetal2010}
Cordeiro, G. M., Ortega, E. M. and Nadarajah, S. (2010). The Kumaraswamy Weibull distribution with application to failure data. {\em Journal of the Franklin Institute}, 347, 1399--429.


\bibitem{Famoyeetal2005}
Famoye, F., Lee, C. and Olumolade, O. (2005). The beta-Weibull distribution. {\em Journal of Statistical Theory and Applications}, 4(2), 121--36.

\bibitem{Kadim2013}
 AlKadim, K. A.  and  Boshi, M. A. (2013). Exponential Pareto distribution. {\em Mathematical Theory and Modeling }, 3, 135--146.

\bibitem{KenneyandKeeping1962}
Kenney, J. and Keeping, E. (1962). Mathematics of Statistics, {\em Volume 1, Princeton}.

\bibitem{Laietal2003}
Lai C. D., Xie, M. and Murthy D. N. P. (2003). A modified Weibull distributions. {\em IEEE Transactions on Reliability}, 52(1), 33--7.

 \bibitem{Bassiouny2015}
 Bassiouny, A. H.  Abdo, N. F.  and Shahen, H. S.  (2015). Exponential Lomax distribution. {\em International Journal of Computer Application}, 13, 24--29.

\bibitem{Lawless2003}
Lawless, J. F. (2003). Statistical Models and Methods for Lifetime Data. {\em John Wiley and Sons, New York}, 20, 1108--1113.

\bibitem{Moors1998}
Moors, J. J. A. (1998). A quantile alternative for kurtosis. {\em The Statistician},  37, 25--32.


\bibitem{Murthyetal2003}
Murthy, D. N. P., Xie, M. and Jiang, R. (2003). Weibull Models. {\em John Wiley and Sons, New York}.

\bibitem{Nadarajahetal2011}
Nadarajah, S., Cordeiro, G.M. and Ortega, E.M.M. (2011). General results for the beta-modified Weibull distribution. {\em Journal of Statistical Computation and Simulation}, 81(10), 1211--32.

\bibitem{PhamandLai2007}
Pham, H. and Lai, C. D. (2007). On recent generalizations of the Weibull distribution. {\em IEEE Transactions on Reliability}, 56, 454--8.

\bibitem{Salmanetal1999}
Salman Suprawhardana M. and Prayoto, Sangadji. (1999). Total time on test plot analysis for mechanical components of the RSG-GAS reactor. {\em Atom Indones}, 25(2), 155--61.

\bibitem{SarhanandZaindin2009}
Sarhan, A. M. and Zaindin, M. (2009). Modified Weibull distribution. {\em Applied Sciences}, 11, 123--136.

\bibitem{SarhanandApaloo2013}
Sarhan, A. M. and Apaloo, J. (2013). Exponentiated modified Weibull extension distribution. {\em Reliability Engineering and System Safety}, 112, 137--144.

\bibitem{Schwarz1978}
Schwarz, G. (1978). Estimating the dimension of a model. {\em Annals of Statistics}, 6, 461--4.

\bibitem{Silvaetal2010}
Silva, G. O., Ortega, E. M. and Cordeiro, G. M. (2010). The beta modified Weibull distribution. {\em Lifetime Data Analysis}, 16, 409--30.


\bibitem{XieandLai1995}
Xie, M. and Lai, C. D. (1995). Reliability analysis using an additive Weibull model with bathtub-shaped failure ratefunction. {\em Reliability Engineering System Safety}, 52, 87--93.


\bibitem{Zwillinger2014}
Zwillinger, D. (2014). Table of integrals, series, and products. {\em Elsevier}.

\bibitem{Weibull1951}
Weibull, W. A. (1951). Statistical distribution function of wide applicability. {\em Journal of Applied Mechanics}, 18, 293--6.

\end{thebibliography}
\end{document}